\documentclass[11pt,reqno,tbtags]{amsart}
\usepackage{amssymb}
\usepackage{natbib}
\bibpunct[, ]{[}{]}{;}{n}{,}{,}

\numberwithin{equation}{section}
\allowdisplaybreaks

\newtheorem{theorem}{Theorem}[section]
\newtheorem{lemma}[theorem]{Lemma}

\newtheorem{corollary}[theorem]{Corollary}
\newtheorem{conjecture}[theorem]{Conjecture}

\theoremstyle{definition}

\newtheorem{definition}{Definition}

\theoremstyle{remark}

\newenvironment{romenumerate}{\begin{enumerate}
 }{\end{enumerate}}

\newcounter{oldenumi}
{\setcounter{oldenumi}{\value{enumi}}
\begin{romenumerate} \setcounter{enumi}{\value{oldenumi}}}
{\end{romenumerate}}

\newcounter{thmenumerate}
\newenvironment{thmenumerate}
{\setcounter{thmenumerate}{0}%
 \def\item{\par
 \refstepcounter{thmenumerate}\textup{(\roman{thmenumerate})\enspace}}
}
{}

\newcounter{xenumerate}   

\newenvironment{proofof}[1]{\noindent {\bf
Proof of #1}.}{\proofbox\par\smallskip\par}

\begingroup
  \divide\count255 by 60
  \multiply\count255 by -60
  \advance\count255 by \time
  \ifnum \count255 < 10 \xdef\klockan{\the\count1.0\the\count255}
  \else\xdef\klockan{\the\count1.\the\count255}\fi
\endgroup

\newcommand{\halmos}{\rule{1ex}{1.4ex}}
\newcommand{\proofbox}{\hspace*{\fill}\mbox{$\halmos$}}

\def\rompar(#1){\textup(#1\textup)}    

\def\xexp(#1){e^{#1}}

\newcommand\R{\mathbb R}

\newcommand\N{\mathbb N}  

\newcommand\Z{\mathbb Z}

\newcounter{CC}

\newcommand\ran{\operatorname{\mathrm ran}}
\newcommand\dev{\operatorname{\mathrm dev}}
\newcommand\var{\operatorname{\mathrm var}}

\newcommand\E{\operatorname{\mathbb E{}}}
\renewcommand\P{\operatorname{\mathbb P{}}}

\newcommand{\lip}{\mbox{\tiny Lip}}

\newcommand\eps{\varepsilon}

\newcommand\cF{\mathcal F}

\newcommand\cL{{\mathcal L}}
\newcommand\cP{\mathcal P}

\def\[#1]{[\![#1]\!]}

\newcommand\dtv{d_{\mathrm{TV}}}
\newcommand\dw{d_{\mathrm{W}}}

\newcommand{\X}{\Omega}
\newcommand{\deq}{:=}
\newcommand{\ed}{\mathcal{E}}
\newcommand{\st}{\,:\,}
\newcommand{\tmix}{t_{{\rm mix}}}
\newcommand{\cn}{\gamma}

\begin{document}
\title
{Concentration of measure and mixing for Markov chains}

\date{1 September 2008} 

\author{Malwina J. Luczak}
\address{Department of Mathematics, London School of Economics,
  Houghton Street, London WC2A 2AE, United Kingdom}
\email{m.j.luczak@lse.ac.uk}
\urladdr{http://www.lse.ac.uk/people/m.j.luczak@lse.ac.uk/}

\keywords{Markov chains, concentration of measure, rapid mixing}
\subjclass[2000]{60C05, 60F05, 60J75}

\begin{abstract} 
We consider Markovian models on graphs with local dynamics. We show
that, under suitable conditions, such Markov chains exhibit both rapid
convergence to equilibrium and strong concentration of measure in the
stationary distribution. We illustrate our results with applications to
some known chains from computer science and statistical mechanics.

\end{abstract}

\maketitle

\section{Introduction}\label{S:intro}

Recent years have witnessed a surge of activity 
in the mathematics of real-world networks, especially the
study of combinatorial and stochastic models. Such networks include,
for instance, the Internet, social networks, and biological networks.
The techniques used to analyse them 
draw from a range of mathematical disciplines, such as graph theory,
probability, statistical physics, analysis.  
Strikingly, random processes with rather similar
characteristics can occur as models of very different real-world settings.

Random networks can often be regarded 
as interacting systems of individuals or
particles. Under certain conditions, there is a law of large numbers,
that is, a large system is 
close to a deterministic process solving a differential equation
derived from the average `drift', with much simpler dynamics.
Further, one may frequently observe {\it chaoticity}, i.e. asymptotic 
approximate independence of particles. 
Unfortunately, it is often difficult to prove the validity of
such approximations,
especially when the random process has an unbounded number of 
components in the limit (e.g. the number of vertices or components of
size $k$ in a graph of size $n$, for $k=1,2, \ldots$, as $n \to \infty$). 

In other instances, it 
may be difficult to establish good rates of
convergence for mean-field approximations, or determine whether the
long-term and equilibrium behaviour of the random process also follows
that of the deterministic system. Furthermore, 
some recent attempts at a more accurate representation of
real networks still await 
any kind of mathematically rigorous analysis.
We would hope that over the coming years, the
intense interest will produce a coherent and widely
applicable theory. However, at present, it often appears that
each new problem defies the existing theory in an interesting way.

In many complex systems, laws of large numbers and high concentration
of measure in equilibrium have been found to co-exist with so-called
{\em rapid
mixing}~\cite{bd97,j98, weitz}, that is
mixing in time $O(n \log n)$, where $n$ is a measure of the system
size. (Traditionally, such a system was considered to be rapidly
mixing if it converged to equilibrium in a time polynomial in $n$, but
currently the term is more and more restricted to the `optimal' mixing
time $O(n \log n)$, see for example~\cite{weitz,dgm00}.)
There are some very notable examples of such behaviour, for instance, the
subcritical Ising model, see~\cite{dlp08,llp08} and references
therein, as well as the discussion in Section~\ref{sec:ising} of this paper.

The purpose of this article is to propose a new method to establish
concentration of measure in complex systems modelled by Markov
chains. We illustrate the technique with an application to a
balls-and-bins model analysed in some
earlier works by this author and McDiarmid, the {\em supermarket
model}~\cite{lm04,lm04b}. Strong concentration of measure for this model, over
long time intervals starting from a given state, as well as in
equilibrium, was established in~\cite{lm04,lm04b} using the underlying structure of the
model that enabled certain functions to be considered as functions of
independent random variables so that the bounded differences method
could be used.

In Section~\ref{sec:conc} of the present article we show that such
concentration of measure inequalities hold more generally, 
with fewer assumptions on the structure of the Markov process involved. Our result is
somewhat related, in spirit, to results (and arguments) in~\cite{lm08},
which establishes transportation cost inequalities for the measure at
time $t$ and the stationary measure of a contracting Markov
chain, assuming transportation cost inequalities for the kernel. However, the technical 
approach adopted here is rather different from~\cite{lm08} -- discrete
and coupling-based rather than functional analytic, 
and, we think, more `hands on' and easier to use in practice (though
our setting is less general than in~\cite{lm08}).
It is striking that our approach, considerably more general than the one
taken in~\cite{lm04}, enables us to improve on the concentration of measure
results proved in~\cite{lm04}. (Accordingly, we could also prove
improved versions of results in~\cite{lm04b}, but we choose not to
pursue this here.) The results in Section~\ref{sec:conc} also
significantly extend Lemma 2.6 in~\cite{llp08}, which bounds the
variance of a real-valued, discrete-time, contracting Markov chain at
time $t$ and in equilibrium.
We hope many more applications for the ideas
presented here will be found in the future.

\medskip

\section{Notation and definitions}

\label{sec:not}

Let $X=(X_t)_{t \in \Z^+}$ be a
discrete-time Markov chain with
a discrete state space $S$ and transition probabilities $P(x,y)$ for $x,y
\in S$, where $\sum_{y \in S} P(x,y) =1$ for each $x \in S$.
We assume that, for every pair of states $x,y \in S$, $P(x,y) > 0$ if
and only if $P(y,x) > 0$. Then we can form an undirected graph
with vertex set $S$ where $\{x,y\}$
is an edge if and only if $P(x,y) > 0$ and $x \not =y$. In general,
our chains may be lazy, that is we can have $P(x,x) > 0$ for some $x \in S$. We assume that the graph is
locally finite, that is, each vertex is adjacent to only finitely many
other vertices. We now endow $S$
with a graph metric $d$ given by $d(x,y) =1$ if $P(x,y) >0$ and $x
\not = y$, and, for all other $x,y$,
$d(x,y)$  the length of the shortest path between $x$ and
$y$ in the graph, which is assumed to be connected.

This kind of setting is natural and many models
in applied probability and combinatorics fit into this framework, 
including those discussed in Section~\ref{sec:examples}.

For each $t \in \Z^+$, $X_t$ may be viewed as a random variable on a measurable space
$(\Omega, \cF)$, where
$$\Omega = \{\omega= (\omega_0, \omega_1, \ldots): \omega_i \in S
\quad \forall i \},$$
and $\cF = \sigma (\cup_{t=0}^{\infty} \cF_t)$, with $\cF_t = \sigma
(X_i: i \le t)$. Then each $X_i$ is the $i$-co-ordinate
projection, that is $X_i (\omega)=\omega_i$ for $i \in \Z^+$. Then the
$\sigma$-fields $\cF_t$ form the natural filtration for the process.

Let $\cP (S)$ be the power set of $S$.
The law of the Markov chain is a probability measure $\P$ on
$(\Omega, \cF)$, and is determined uniquely by the transition matrix
$P$ together with a probability measure $\mu$ on $(S, \cP
(S))$ that gives the law of
the initial state $X_0$, according to
\begin{eqnarray*}
\P (\{ \omega: \omega_j = x_j: j \le i\}) = \mu (\{x_0\})
\prod_{j=0}^{i-1} P(x_j,x_{j+1}),
\end{eqnarray*}
for each $x_0, \ldots, x_i \in S$, for each $i \in \Z^+$.
This gives the law of $(X_t)$  conditional on $\cL (X_0) =
\mu$, and will be denoted by  $\P_{\mu}$ in what follows. Let
 $P^t(x,y)$ be the $t$-step transition probability from $x$ to
$y$, given inductively by
$$P^t (x,y) = \sum_{z \in S} P^{t-1} (x,z) P(z,y).$$
Then $\P_{\mu} (X_t \in A) = (\mu P^t)(A)$ for $A \subseteq S$.

Let $\E_{\mu}$ denote the expectation operator corresponding to
$\P_{\mu}$.
For $t \in \Z^+$ and $f: S \to \R$, define the function $P^t f$ by
$$(P^t f)(x) = \sum_y P^t(x,y) f(y),  \quad x \in S.$$
In other words, $(P^t f)(x) = \E_{\delta_x} [f(X_t)]= (\delta_x
P^t)(f)$, the expected value of $f(X_t)$ at time $t$ conditional on
the Markov process starting at $x$, i.e. the expectation of the
function $f$ with respect to measure $\delta_x P^t$. In general, we
write $\E_{\mu} [f(X_t)] = (\mu P^t)(f)$.

A real-valued function $f$ on $S$ is said to be Lipschitz (or
1-Lipschtitz) if
$$\parallel f \parallel_{\lip}= \sup_{x \not = y}
\frac{|f(x)-f(y)|}{d(x,y)} \le 1.$$
Here, equivalently, we only need to consider vertices at distance 1,
so $f$ is Lipschitz if and only if 
$\sup_{x,y:d(x,y)=1} |f(x)-f(y)| \le 1$.

Given a probability measure
$\mu$ on $(S, \cP (S))$ and an $S$-valued random variable $X$ with law
$\cL (X)=\mu$, we say that $\mu$ or $X$ has {\em normal 
concentration} if there exist constants $C,c >0$ such that, for every
$u > 0$, uniformly over 1-Lipschitz functions $f:S \to \R$,
\begin{equation}
\label{def-norm-conc}
\mu (|f(X)- \mu (f)| \ge u) \le C e^{-cu^2}.
\end{equation}
We say that $\mu$ or $X$ has {\em exponential concentration} if 
there exist constants $C,c >0$ such that, for every
$u > 0$, uniformly over 1-Lipschitz functions $f:S \to \R$,
\begin{equation}
\label{def-exp-conc}
\mu (|f(X)- \mu (f)| \ge u) \le C e^{-cu}.
\end{equation}
These definitions are closely related to the notions used by
Ledoux~\cite{ledoux}.

In Section~\ref{sec:conc} we shall give conditions under which a
discrete-time Markov chain $(X_t)$ exhibits  normal 
concentration of measure over long time intervals and in equilibrium.

For probability measures $\mu_1, \mu_2$ on $(S,\P(S))$, the
{\it total variation distance} between $\mu_1$ and $\mu_2$ is given by
$$ \dtv (\mu_1, \mu_2 ) = \frac12 \sum_{x \in S } | \mu_1 (x) - \mu_2
(x)|= \sup_{A \subseteq S} |\mu_1 (A) - \mu_2 (A)|.$$
It is well known
 that the total variation distance satisfies
$$ \dtv (\mu_1, \mu_2) = \inf_{\pi} \pi (X \not = Y),$$
where the infimum is over all couplings $\pi = \cL (X,Y)$ of
$S$-valued random variables $X,Y$ such that the marginals are 
$\cL (X) = \mu_1$ and $\cL (Y) = \mu_2$. 

The {\em Wasserstein
distance} between probability measures $\mu_1$ and $\mu_2$ is defined as
$$\dw (\mu_1 , \mu_2) = \sup_f \left | \int f d\mu_1 - \int f d \mu_2 \right
|= \sup_f |\mu_1 (f)-\mu_2 (f)|,$$
where the supremum is over all measurable 1-Lipschitz
functions $f: S \to \R$.
By the Kantorovich -- Rubinstein theorem (see~\cite{d89}, Section 11.8),
$$\dw = \inf_{\pi} \{ \pi [d (X,Y)]: \cL(X)=\mu_1, \cL(Y)
= \mu_2 \},$$ 
where the infimum is taken over all couplings $\pi$ on $S
\times S$ with
marginals $\mu_1$ and $\mu_2$, and we write $\pi [d(X,Y)]$ for the
expectation of $d(X,Y)$ under the coupling $\pi$. It is well known that the Wasserstein
distance metrises weak
convergence in spaces of bounded diameter. Also, since the discrete
space $(S, \cP(S))$
is necessarily complete and separable, so is the space of probability
measures on $(S, \cP(S))$ metrised by the Wasserstein distance.
See~\cite{r91} for
detailed discussions of various metrics on probability measures and
relationships between them.

\medskip

\section{Examples of rapid mixing and concentration}

In this section we give some examples of known Markov chains
exhibiting both concentration of measure in equilibrium and rapid mixing.

\label{sec:examples}

\subsection{Mean-field Ising model}

\label{sec:ising}

Let $G = (V, {\mathcal E})$ be a finite graph.  Elements of the state
space $S \deq \{-1,1\}^V$ will be called \emph{configurations}, and 
for $\sigma \in S$, the value $\sigma(v)$ will be called the 
\emph{spin} at $v$.  The \emph{nearest-neighbour energy} $H(\sigma)$ of
a configuration $\sigma \in \{-1,1\}^V$ is defined by
\begin{equation} \label{eq.energy}
  H(\sigma) 
    \deq -\sum_{\substack{v,w \in V,\\ v \sim w}} 
      J(v,w) \sigma(v)\sigma(w),   
\end{equation}
where $w \sim v$ means that $\{w,v\} \in \ed$.  The parameters
$J(v,w)$ measure the interaction strength between vertices;
we will always take $J(v,w) \equiv J$, where $J$ is a positive
constant.

For $\beta \geq 0$, the \emph{Ising model} on the graph $G$ with
parameter $\beta$ is the probability measure $\pi$
on $S$ given by 
\begin{equation} \label{eq.gibbs}
  \pi(\sigma) = \frac{e^{-\beta H(\sigma)}}{Z(\beta)},
\end{equation}
where $Z(\beta) = \sum_{\sigma \in \X} e^{-\beta H(\sigma)}$
is a normalising constant.

The parameter $\beta$ is interpreted physically as the inverse 
of temperature, and measures the influence of the energy
function $H$ on the probability distribution.  
At \emph{infinite temperature}, corresponding to $\beta = 0$,
the measure $\pi$ is uniform over $S$ and the random
variables $\{ \sigma(v) \}_{v \in V}$ are independent.

The (single-site) \emph{Glauber dynamics} for $\pi$ is the Markov
chain $(X_t)$ on $S$ with transitions as follows. When at $\sigma$, a
vertex $v$ is chosen uniformly at random from $V$, and a new
configuration is generated from $\pi$ conditioned on the set  
\begin{equation*}
  \{ \eta \in S \st \eta(w) = \sigma(w), \; w \neq v \}.
\end{equation*}
In other words, if vertex $v$ is selected, the new configuration will 
agree with $\sigma$ everywhere except possibly at $v$, and at
$v$ the spin is $+1$ with probability 
\begin{equation} \label{eq.glauber}
  p(\sigma ; v) 
    \deq \frac{ e^{\beta M^v(\sigma)} }{ 
                e^{\beta M^v(\sigma)} 
		  + e^{-\beta M^v(\sigma)} },
\end{equation}
where $M^v(\sigma) \deq J\sum_{w \,:\, w\sim v} \sigma(w)$.
Evidently, the distribution of the new spin at $v$ depends only on the 
current spins at the neighbours of $v$.  It is easily seen
that $(X_t)$ is reversible with respect to the measure
$\pi$ in \eqref{eq.gibbs}, which is thus its stationary measure.

Given a sequence  $G_n = (V_n, E_n)$ of graphs, write $\pi_n$ for
the Ising measure and $(X^{(n)}_t)$ for the
Glauber dynamics on $G_n$. 
For a given configuration $\sigma \in S_n$, let $\cL (X^{(n)}_t,
\sigma)$ denote the law of $X^{(n)}_t$ starting from $\sigma$.
The worst-case distance to stationarity of the Glauber
dynamics chain after $t$ steps is  
\begin{equation} \label{eq.disttostat}
  d_n(t) 
    \deq \max_{\sigma \in S_n} \dtv( \cL ( X^{(n)}_t, \sigma),
      \pi_n).
\end{equation}
The \emph{mixing time}
$\tmix(n)$ is defined as 
\begin{equation} \label{eq.tmix-defn}
  \tmix(n) \deq \min\{t \,:\, d_n(t) \leq 1/4 \}.
\end{equation}
Note that $\tmix (n)$ is finite for each fixed $n$ since, 
by the convergence theorem for ergodic Markov chains,
$d_n(t) \rightarrow 0$ as $t \rightarrow \infty$.
Nevertheless, $\tmix (n)$ will in general tend to infinity with $n$.
It is natural to ask about the growth rate of the sequence
$\tmix(n)$.

\begin{definition}
\label{defn-cutoff}
The Glauber dynamics is said to exhibit a \emph{cut-off}
at $\{t_n\}$ with \emph{window size} $\{w_n\}$ if $w_n = o(t_n)$ and
\begin{align*}
  \lim_{\cn \rightarrow \infty} \liminf_{n \rightarrow \infty} 
    d_n(t_n - \cn w_n) & = 1, \\
  \lim_{\cn \rightarrow \infty} \limsup_{n \rightarrow \infty}
    d_n(t_n + \cn w_n) & = 0.
\end{align*}
\end{definition}

Informally, a cut-off is a sharp threshold for mixing.
For background on mixing times and cut-off, see \cite{lpw}.

Here we consider the mean-field case, taking $G_n$ to be $K_n$, the
complete graph on $n$ vertices. 
That is, the vertex set is $V_n = \{1,2,\ldots,n\}$, and the edge set
${\mathcal E}_n$ contains all $\binom{n}{2}$ pairs $\{i, j\}$ for 
$1 \leq i < j \leq n$.  We take the interaction parameter
$J$ to be $1/n$; in this case, the Ising measure
$\pi$ on $\{-1,1\}^n$ is given by
\begin{equation} \label{eq.gibbs-cw}
  \pi(\sigma) 
    = \pi_n (\sigma) 
    = \frac{1}{Z(\beta)} 
      \exp\left( \frac{\beta}{n}\sum_{1 \leq i < j \leq n} 
        \sigma(i)\sigma(j) \right).
\end{equation}
In the physics literature, this is usually referred to as the
\emph{Curie-Weiss} model.  
To put this into the framework introduced in Section~\ref{sec:not},
the state space $S$ consists of all $n$-vectors with components taking
values in $\{-1,1\}$, and two vectors are adjacent if they differ in
exactly one co-ordinate.

It is a consequence of the Dobrushin-Shlosman uniqueness criterion
that $\tmix(n) = O(n\log n)$ when $\beta < 1$; see \cite{AH}. (See
also~\cite{bd97,weitz}). We shall see in Section~\ref{sec:conc} that, in the
same regime, the stationary measure $\pi$ (the Gibbs measure) exhibits
normal concentration of measure 
for Lipschitz functions in the following sense. Let $X^{(n)}$ be a
stationary version of $X_t^{(n)}$. Then, for some constants
$c,C>0$, for all $u > 0$,
\begin{equation}
\label{eqn.ising-conc}
\P_\pi (|f(X^{(n)}) - \E_\pi (f(X^{(n)})) \ge u) \le Ce^{-u^2/cn},
\end{equation}
uniformly over all 1-Lipschitz  functions on $S$ and over all $n$. 
Thinking about~(\ref{eqn.ising-conc}) simply as a statement about the
measure $\pi$ without any mention of the process $X_t^{(n)}$, we can
also rewrite it in the form
$$\pi (\{\sigma: |f(\sigma)-\pi (f)| \ge u \}) \le C e^{-u^2/cn}.$$
Inequality~(\ref{eqn.ising-conc}) will follow from
Theorem~\ref{thm.conca-general}~(i), and is an improvement on
Proposition 2.7 in~\cite{llp08}.

More precise results
about the speed of mixing for $\beta < 1$ can be found
in~\cite{llp08}, where the occurrence of a cut-off is established.
 The following is Theorem 1
from~\cite{llp08}:

\begin{theorem} \label{thm.hightemp}
  Suppose that $\beta < 1$.  The Glauber dynamics for the Ising
  model on $K_n$ has a cut-off at $t_n = [2(1-\beta)]^{-1}n\log n$ with
  window size $n$. 
\end{theorem}

It is also easy to show, using
the concentration of the Gibbs measure and the method used to prove
Theorem 1.4 in~\cite{lm04b}, that asymptotically the spin values in a
bounded set of vertices become almost independent. (In the language
of~\cite{weitz} -- see also references therein -- this corresponds to the
decay of correlations or spatial mixing.)

On the other hand, in the case $\beta \ge 1$, there is no rapid
mixing, and no cut-off (see~\cite{llp08,dlp08} and references therein): $\tmix (n)$ is of the
order $n^{3/2}$ when 
$\beta =1$ and is exponential in $n$ when $\beta > 1$. For the same
range of $\beta$, the Gibbs measure fails to exhibit normal
concentration. 

In particular, consider the function $m: S \to \R$
given by $m(\sigma)= \sum_{i=1}^n \sigma (i)$, the {\em
magnetisation}; it is easy to see that $\frac12 m$ is 1-Lipschitz, and
$\E_{\pi} (m(X)) = \pi (m) =0$. However, when $\beta >1$, then
there is a constant $c >0$ such that 
$$\pi (\{\sigma: m(\sigma) \ge cn \})=\pi (\{\sigma: m(\sigma) \le - cn
\})\ge 1/4,$$
i.e. $m(X)$ is bi-modal for $\beta > 1$. 
While there is no bi-modality in the case $\beta =1$, it is easy to
calculate directly that $m(X)$ is not concentrated in the sense of~(\ref{eqn.ising-conc}). 
Further, for $\beta \ge 1$,
the spins of vertices are no longer approximately independent for
large $n$.

\medskip

\subsection{Supermarket model}

\label{sec:balls-and-bins}

Consider the following
well-known queueing model with $n$ separate
queues, each with a single server. Customers arrive
into the system in a Poisson process at rate $\lambda n$, where
$0<\lambda <1$ is a constant. Upon arrival each customer chooses
$d$ queues uniformly at random with replacement, and joins a
shortest queue amongst those chosen (where she breaks ties by
choosing the first of the shortest queues in the list of $d$).
Here $d$ is a fixed positive integer. Customers are served
according to the first-come first-served discipline. Service times
are independent exponentially distributed random variables with
mean 1.

A number of authors have studied this model, as well as its
extension to a Jackson network
setting~\cite{g00,g04,lm04,lm04b,ln04,ms99,m96a,m01,vdk96}.

For instance, it is shown by Graham in~\cite{g00} that the system is {\em
chaotic}, provided that it starts close to a suitable
deterministic initial state, or is in equilibrium. This means that
the paths of members of any fixed finite subset of queues are
asymptotically independent of one another, uniformly on bounded
time intervals. This result implies a law of large numbers for the
time evolution of the proportion of queues of different lengths,
that is, for the empirical measure on path space~\cite{g00}. In
particular, for each fixed positive integer $k_0$, as $n$
tends to infinity the proportion of queues with length at least
$k_0$ converges weakly
(when the infinite-dimensional state space is endowed with the
product topology) to a function $v_t(k_0)$, where $v_t(0)=1$ for
all $t \ge 0$ and $(v_t(k): k \in \N)$ is the unique solution
to the system of differential equations
\begin{equation}
\label{eqn.lln}
 {{dv_t(k)} \over {dt}} = \lambda (v_t(k-1)^d-v_t(k)^d) - (v_t(k) -
v_t(k+1))
\end{equation}
for $k \in \N$. Here one needs to assume appropriate initial
conditions $(v_0(k): k \in \N)$ such that $1 \ge v_0(1) \ge
v_0(2) \geq \cdots \ge 0$. Further, again for a fixed positive
integer $k_0$, as $n$ tends to infinity, in the equilibrium
distribution this proportion converges in probability to
$\lambda^{1+d+\cdots +d^{k_0-1}}$, and thus 
the probability that a given queue has length at least $k_0$ also
converges to $\lambda^{1+d+\cdots +d^{k_0-1}}$.

Although the above results refer only to fixed queue length $k_0$ and
bounded time intervals, they suggest that when $d \geq 2$, in
equilibrium the maximum queue length may usually be $O(\log\log n)$.
Indeed, one of the contributions of~\cite{lm04}
 is to show that this is indeed the case, and
to give precise results on the behaviour of the maximum queue
length. In particular, it turns out that when $d \ge 2$, with
probability tending to 1 as $n \rightarrow \infty$, in the
equilibrium distribution the maximum queue length takes at most
two values; and these values are $\log\log n/ \log d +O(1)$. Along the way, it
is also shown in~\cite{lm04}
that the system is rapidly mixing, that is the distribution
settles down quickly to the equilibrium distribution. In this context,
`quickly' will mean `in time $O(\log n)$, as this is a continuous
time process with events happening at rate $n$, and so $O(\log n)$
corresponds to $O(n \log n)$ steps of the discrete-time jump chain. It is further
established in~\cite{lm04} that the equilibrium measure is strongly concentrated.

Another
natural question concerns fluctuations when in the equilibrium
distribution: how long does it take to see large deviations of the
maximum queue length from its stationary median? An answer is provided
in~\cite{lm04} by establishing strong concentration estimates (for
Lipschitz functions of the queue lengths vector) over time
intervals of length polynomial in $n$. The techniques in~\cite{lm04} are partly
combinatorial, and are used also in~\cite{lm03} and~\cite{lm04b}.
In particular, in~\cite{lm04b}, the concentration estimates
obtained in~\cite{lm04} are used to establish quantitative results on the convergence
of the distribution of a queue length and on `propagation of
chaos'.

Let us start by discussing the rapid mixing results known for the
supermarket model.
In~\cite{lm04} two rapid mixing results are established, one in terms of the Wasserstein
distance and one in terms of the total variation distance. Unlike for
the Ising model in Section~\ref{sec:ising}, it turns out to be
inappropriate to be looking at the worst-case mixing time, that
is the supremum of the mixing times over all possible starting
states. In the present case, this quantity is unbounded: the state
space is unbounded, and the time to equilibrium from states $x$ with
the total number of customers $\parallel x \parallel_1 =k \gg n$ is of
the order at least $k$. Then the best
one can do is to obtain good upper bounds on the mixing time for copies of the Markov
chain starting from nice states -- that is, states where the queues
are not too `over-loaded'. This is made more precise below.

Let $X_t^{(n)}$ or $X_t$ be the queue-lengths vector
$(X^{(n)}_t(1), \ldots, X^{(n)}_t (n))$ in the supermarket model with
$n$ servers. For a positive integer $n$,
$(X_t^{(n)})$ is an ergodic continuous-time Markov chain, with a unique
distribution $\pi^{(n)}$ or $\pi$.

For any given state $x$ write $\cL (X_t^{(n)}, x)$ to denote the law
of $X_t^{(n)}$ given $X_0^{(n)}=x$. Also, for $\epsilon > 0$, the
mixing time $\tau^{(n)}(\epsilon,x)$ starting from $x$ us defined by
$$\tau^{(n)} (\epsilon,x) = \inf \{t \ge 0: \dtv (\cL(X_t^{(n)},x),
{\pi}^{(n)}) \le \epsilon \}.$$

The result below, Theorem 1.1 in~\cite{lm04},
shows that starting from an initial state in which
the queues are not too long, the mixing time is small. In particular,
if $\epsilon > 0$ is fixed and ${\bf 0}$ denotes the all-zero
$n$-vector, then $\tau^{(n)}(\epsilon, {\bf 0})$ is $O(\log n)$.

\begin{theorem} \label{thm.stat}
Let $0<\lambda<1$ and let $d$ be a fixed positive integer.
For each constant $c>0$ there exists a constant $\eta >0$ 
such that the following holds for each positive integer $n$.
Consider any distribution of the initial queue-lengths vector
$X^{(n)}_0$, and for each time $t \geq 0$ let
$$\delta_{n,t} = \P(|X^{(n)}_0| > c n) + \P(M^{(n)}_0 > \eta t).$$
Then
$$\dtv ({\mathcal L}({X}^{(n)}_t),{\pi}^{(n)}) \leq
n e^{-\eta t} + 2 e^{-\eta n} + \delta_{n,t}.$$
\end{theorem}

The $O(\log n)$ upper bound on the mixing time $\tau$ is of the
right order. Indeed, it is also proven in~\cite{lm04} that,
for a suitable constant $\theta >0$, if $t \leq \theta \log n$ then
\begin{equation} \label{mixlower}
\dtv ({\mathcal L}({X}^{(n)}_t), \pi^{(n)}) = 1-e^{-\Omega
(\log^2 n)}.
\end{equation}
Thus $\tau^{(n)}(\epsilon,{\bf 0})$ is $\Theta(\log n)$ as long as
both $\epsilon^{-1}$ and $(1-\epsilon)^{-1}$ are bounded
polynomially in $n$.

It would be interesting to consider the mixing times more precisely,
to establish whether the supermarket model exhibits a cut-off. Again,
here we should not be considering the worst-case mixing time, but
rather the worst case over a subset of `good' initial states, which are
states where the total number of customers is not too large and the
maximum queue not too long. Also, to bring the supermarket model into
the discrete framework of Section~\ref{sec:not}, let us consider the jump
chain of the supermarket model. We
shall denote the jump chain by $\hat{X}_t^{(n)}$ or $\hat{X}_t$ 
in what follows, and its stationary measure by $\hat{\pi}^{(n)}$ or
$\hat{\pi}$.

The transition probabilities of the
jump chain are as follows. Given the state at time $t$ is $x$, the
next event is an arrival with probability $\lambda/(\lambda +1)$ and is
a {\em potential} departure with probability $1/(\lambda +1)$. Here
`potential' means that it may be a departure or no change of state at
all. Given that the next event is an arrival, the queue to which the
new customer is sent is determined by selecting a uniformly random
$d$-tuple of queues and directing the customer to a shortest queue
among those chosen, in the same way as for the continuous-time process. Given that the
next event is a potential departure, the departure queue is chosen
uniformly at random from among all $n$ queues. Then a customer will
depart if the selected queue is non-empty; otherwise, nothing
happens. It is easy to adapt the proofs in~\cite{lm04} (where the
arguments are, in fact, based on analysing the jump chain) to show
that Theorem~\ref{thm.stat} implies mixing in time of the order $O(n
\log n)$ from initial states $x$ such that $\parallel x \parallel_1
=O(n)$ and $\parallel x \parallel_{\infty}=O(\log n)$.

Accordingly, we make the following conjecture:

\begin{conjecture}
Let $c$ be a positive constant, and let $S_0^{(n)}$ be the set of all
queue-lengths vectors $x$ in the $n$ server supermarket model 
such that $\parallel x \parallel_1 \le c n$ and  $\parallel x
\parallel_{\infty} \le c \log n$. Let $\epsilon > 0$, and let
$$d_n(\epsilon,t) = \sup_{x \in S_0^{(n)}} \dtv (\cL (\hat{X}_t^{(n)}, x),
{\hat \pi}^{(n)}).$$
Then $d_n (\epsilon,t)$ has a cut-off in the sense of
Definition~\ref{defn-cutoff}, with window size $n$.
\end{conjecture}
Our conjecture appears supported by some simulation results. Also it
is supported by Conjecture 1 from~\cite{llp08}, which states that
the Glauber dynamics for the Ising model on transitive graphs $G_n$ has a cutoff if the
mixing time is $O(n \log n)$. The jump chain of the supermarket
process is of a similar type to Glauber dynamics in that it makes only
local transitions, and has mixing time of the order $O(n \log n)$,
starting
from good initial states. Also, it has a lot of symmetry -- its
stationary distribution is exchangeable. Thus the supermarket chain
appears a good candidate for cut-off, though proving it may not be easy. 

More generally, perhaps cut-off can be proven to be a phenomenon that
also co-occurs with rapid  mixing and concentration of measure in
equilibrium much more widely, in the context of Markov chains whose
jumps are suitably local.

\medskip

In~\cite{lm04}, the authors upper bound mixing in terms of the total
variation distance by first upper bounding the Wasserstein distance between the
distribution of the process at time $t$ and the stationary distribution.
The following result is Lemma 2.1 in~\cite{lm04}.

\begin{theorem} \label{thm.statW}
Let $0<\lambda<1$ and let $d$ be a fixed positive integer. For
each constant $c>\frac{\lambda}{1-\lambda}$
there exists a constant $\eta >0$ 
such that the following holds for each positive integer $n$. Let
$M$ denote the stationary maximum queue length.
Consider any distribution of the initial queue-lengths vector
$X_0$ such that $|X_0|$ has finite mean. For each time
$t \geq 0$ let
$$\delta_{n,t} = 2 \E[|X_0| {\bf 1}_{|X_0|>cn}]+
2cn \ \P(M_0 > \eta t).$$
Then
$$\dw ({\mathcal L}({X}_t), \pi) \leq
n e^{-\eta t} + 2c n \P_{\pi}(M> \eta t)+ 
2e^{-\eta n} + \delta_{n,t}.$$
\end{theorem}

The upper bounds on the Wasserstein and total variation distance, and
thus on the mixing time, are proven in~\cite{lm04} by means of a
monotone coupling. The coupling takes two
copies of the queueing process starting in adjacent states (that is,
states differing in one customer in one queue) and couples
their paths together in such a way that the $\ell_1$-distance between
them is non-increasing (and so always stays equal to 1 until the
processes coalesce). Furthermore, the coupling is such that with 
high probability the $\ell_1$-distance rapidly becomes 
0.  
The coupling is
then extended to all pairs of starting states with not too many
customers in queues using the fact that the Wasserstein distance is a
metric on the space of probability measures, or 
a {\em path-coupling} argument~\cite{bd97}.

The property that the $\ell_1$-distance is non-increasing in the
coupling in~\cite{lm04} is very
strong and not commonly encountered in path-coupling scenarios. This
property is exploited in~\cite{lm04} to prove strong
concentration of measure for the supermarket process, starting from a fixed (or
highly concentrated state) for a long time interval. The following is
Lemma 4.3 in~\cite{lm04}.

\begin{lemma}
\label{lem.conc} There is a constant $c >0$ such that the
following holds. Let $n \ge 2$ be an 
integer and let $f$ be a
1-Lipschitz function on the state space (set of all queue lengths
vectors) $S$. Let also $x_0 \in S$
and assume that the queue-lengths process $(X_t)$ satisfies $X_0 =
x_0$ a.s. Let $\mu_t = \E_{\delta_{x_0}} [f(X_t)]$.
Then for 
all times $t> 0$ and all $u \geq 0$,
\begin{equation} \label{eqn.extra2}
\P_{\delta_{x_0}} (|f(X_t) - \mu_t| \geq u) \leq n e^{-\frac{c u^2}{nt+u}}.
\end{equation}
\end{lemma}

Lemma 4.3 in~\cite{lm04} is proven by observing that the supermarket
process can be `simulated' by two independent Poisson processes, the
arrivals process (with rate $\lambda n$) and the (potential) departure
process (with rate $n$), together with
corresponding independent choices of queues ($d$ independent uniformly
random choices for each event in the arrivals process, and one
uniformly random choice in the departures process). One then conditions on
the number of events in the interval $[0,t]$, and then the state at time $t$ is
conditionally determined by a finite family of independent random
variables. In other words, the argument is, just like most of the
other arguments in~\cite{lm04}, based on studying the jump chain
$(\hat{X}_t)$, although this is not made explicit therein.

The non-increasing distance coupling property is used to show that a
Lipschitz function of the queue lengths vector must satisfy a bounded
differences condition, so that the discrete bounded differences
inequality can be applied to show concentration of measure for
Lipschitz functions in the conditional space. The proof is then
completed by deconditioning.

The rapid mixing result can be combined with the long-term
concentration of measure result to
prove concentration of measure in equilibrium for Lipschitz functions
of the queue-lengths vector. The following is Lemma 4.1 in~\cite{lm04}.

\begin{lemma} \label{lem.conca}
There is a constant $c > 0$ such that the following holds. Let $n \ge 2$
be an integer 
and consider the $n$-queue system. Let the queue-lengths vector
$Y$ have the equilibrium distribution. Let $f$ be a 1-Lipschitz
function on $S$. Then for each 
$u \geq 0$
\begin{equation}
\label{eqn.extra2a}
\P_{\pi} \left( |f(Y) - \E_{\pi}[f(Y)]| \geq u \right)
\leq n e^{-c u/ n^{\frac12}}.
\end{equation}
\end{lemma}

Lemmas~\ref{lem.conc} and~\ref{lem.conca} prove strong concentration
of measure -- normal concentration for small deviations and
exponential concentration for larger deviations in the case of
starting from a fixed state, and 
exponential concentration in equilibrium. The factor $n$ in the bound on
the right-hand sides of both~(\ref{eqn.extra2})
and~(\ref{eqn.extra2a}) is a limitation of the technique and not the
right answer. It is natural to expect the truth to be a lot better --
that it can be replaced by a constant. In Section~\ref{sec:conc} we
develop concentration inequalities that achieve that. Although we work
with the discrete-time jump chain, it is easy to see that our results
apply also to the continuous time chain. One further advantage of our
inequalities is that they apply to other settings -- for instance where
rapid mixing is established by a coupling, but the
coupling does not have additional useful
properties such as the non-increasing Wasserstein distance.

Even so Lemmas~\ref{lem.conc} and~\ref{lem.conca} are quite powerful.
We now explore, briefly, some results concerning the queue lengths in
the supermarket model in equilibrium that can be obtained using
Lemma~\ref{lem.conca}. The following is Lemma 4.2 in~\cite{lm04}. (We
drop the subscript $\pi$ to lighten up the notation.)

\begin{lemma} \label{lem.concb}
Consider the $n$-queue system, and let the queue-lengths vector
$Y$ have the equilibrium distribution. For each non-negative integer
$k$, let $\ell (k,y)$ denote the
number of queues of length at least $k$ in state $y$. Also, for each non-negative
integer $k$,
let $\ell(k) = \E [\ell(k,Y)]$.  Then for any constant $c>0$,
$$\P (\sup_{k} |\ell(k,Y)- \ell(k)| \geq c n^{\frac12} \log^2 n) =
e^{-\Omega(\log^2 n)}.$$
Also, there exists a constant $c >0$ such that
$$\sup_k \P (|\ell(k,Y)- \ell(k)| \geq c n^{\frac12} \log n) =o(1).$$
Furthermore, for each integer $r \ge 2$
$$\sup_k |\E  [\ell(k,Y)^r] - \ell(k)^r| = O(n^{r-1}\log^2 n).$$
\end{lemma}

Lemma 5.1 in~\cite{lm04}, stated below, yields further precise information about the
equilibrium behaviour, over long time intervals.

\begin{lemma}
\label{lem.customers.2} 
Let $K > 0$ be an arbitrary constant and let $\tau =n^K$. 
Let $(Y_t)$ be in equilibrium and let
$c>0$ be a constant. Let $B_{\tau}$ be the event that for all
times $t$ with $0 \leq t \leq \tau$
$$\sup_i |\ell(i,Y_t)- n\lambda^{1+d+ \cdots+d^{i-1}}| \le c n^{1/2}\log^2 n.$$
Then $\P (\overline{B_{\tau}})\le e^{-\Omega (\log^2 n)}$.
\end{lemma}

In~\cite{lm04}, Lemma 5.1 is used to prove two-point concentration for
the stationary maximum queue length and its concentration on only a
constant number of values over long time intervals. This is Theorem
1.3 in~\cite{lm04}:

\begin{theorem} \label{thm.dis2}
Let $0<\lambda<1$ and let $d \geq 2$ be an integer. Then there
exists an integer-valued function $m_d=m_d(n) = {\log \log n}/{\log
d} +O(1)$ such that the following holds. For each positive integer
$n$, suppose that the queue-lengths vector ${Y}_0^{(n)}$ is in the
stationary distribution (and thus so is the maximum queue length
$M^{(n)}_t$). Then for each time $t \geq 0$, $M_t^{(n)}$ is
$m_d(n)$ or $m_d(n)-1$ with probability tending to 1 as $n \to \infty$; and further, for any constant $K>0$
there exists $c=c(K)$ such that, with probability tending to 1 as $n
\to \infty$,
\begin{equation} \label{eqn.dis2a}
\max_{0 \le t \le n^K} |M^{(n)}_t-\log \log n/\log d| \le c.
\end{equation}
\end{theorem}
The functions $m_2(n)$, $m_3(n)$, ...  may be defined as follows.
For $d=2,3,\ldots$ let $i_d(n)$ be the least integer $i$ such that
$\lambda^{\frac{d^i -1}{d-1}} < n^{-\frac12} \log^2 n$. Then we let
$m_2(n)=i_2(n)+1$, and for $d \geq 3$ let $m_d(n)=i_d(n)$. (As we have
seen, with high probability the proportion
of queues of length at least $i$ is close to $\lambda^{\frac{d^i
-1}{d-1}}$.)

Also, equation~(37) in~\cite{lm04} shows that, for $r=O(\log n)$,
\begin{equation}
\label{eq.max-queue}
\P (M \ge m_d (n) + r) \le e^{-c r \log n},
\end{equation}
for a constant $c>0$.

\medskip

In~\cite{lm04b}, strong concentration of measure results from~\cite{lm04} are used
to show that in equilibrium the distribution of a
typical queue length converges to an explicit limiting distribution
and provide explicit convergence rates. 
Let   $Y^{(n)}(1)$ denote the equilibrium length of of queue 1. (Note
that the equilibrium distribution is exchangeable.) The following is
Theorem 1.1 in~\cite{lm04b}.
   Let ${\cL}_{\lambda,d}$ denote the law of a random variable $Y$
such that $\P (Y \ge k) = v(k)$, where $v(k)=
 \lambda^{(d^k-1)/(d-1)}$ for each $k=0,1,\ldots$.
Note that $\P(Y^{(n)}(1) \geq 1) = \lambda = v(1)$.

   \begin{theorem} \label{thm.distql}
   For each positive integer $n$ let $Y^{(n)}$ be a queue-lengths
   $n$-vector in equilibrium, and consider the length $Y^{(n)}(1)$ of
   queue~1. Then 
$$\dtv ({\cL}(Y^{(n)}(1)), {\cL}_{\lambda,d})$$
   is of order $n^{-1}$ up to logarithmic factors.
   \end{theorem}
    In fact, it is proven in~\cite{lm04b} that the above total variation distance is
   $o(n^{-1} \log^3n)$ and is $\Omega(n^{-1})$.  Also, the following
holds (Corollary 1.2 in~\cite{lm04b}).
   \begin{corollary} \label{cor.distql}
   For each positive integer $k$, the difference between the $k$th moment
   $\E[Y^{(n)}(1)^k]$ and the $k$th moment of ${\cL}_{\lambda,d}$
   is of order $n^{-1}$ up to logarithmic factors.
   \end{corollary}

  The above results concern the distribution of a single queue
   length. One may also consider collections of queues and
   chaoticity. The terms `chaoticity' and `propagation of chaos'
come from statistical
   physics~\cite{kac}, and the original motivation was the evolution of
   particles in physical systems. The subject has since then received
   considerable attention, especially following the ground-breaking work of Sznitman~\cite{s91}.

The result below (Theorem 1.4 in~\cite{lm04b}) establishes chaoticity for the
   supermarket model in equilibrium.
   We see that for fixed $r$ the total variation distance between
   the joint law of $r$ queue lengths and the product law is at most
   $O(n^{-1})$, up to logarithmic factors.  More precisely and more generally
   we have:
   \begin{theorem} \label{thm.chaos1}
   For each positive integer $n$, let $Y^{(n)}$ be a queue-lengths
   $n$-vector in equilibrium. Then, uniformly over all positive
   integers $r \leq n$, the total variation distance between the
   joint law of $Y^{(n)}(1), \ldots, Y^{(n)}(r)$ and the product law
   ${\cL}(Y^{(n)}(1))^{\otimes r}$
   is at most
   $O(n^{-1} \log^2 n (2 \log \log n)^r)$;
   and the total variation distance between the
   joint law of $Y^{(n)}(1), \ldots, Y^{(n)}(r)$ and
   the limiting product law ${\cL}_{\lambda,d}^{\otimes r}$
   is at most
   $O(n^{-1} \log^2 n (2 \log \log n)^{r+1})$.
   \end{theorem}

Analogous time-dependent results (away from equilibrium) are also
given in~\cite{lm04b} -- proven using Lemma~\ref{lem.conc} above
(Lemma 4.3 in~\cite{lm04}) but we omit them here for the sake of
brevity. Let us mention that the arguments used in~\cite{lm04b} to
prove Theorems 1.1 and 1.4 (Theorems~\ref{thm.distql}
and~\ref{thm.chaos1} above) are quite generic and would apply in many
other settings. The main property needed is concentration
of measure for Lipschitz functions of the state vector, the polynomial
form of the generator of the Markov process, and, in the case of
Theorem 1.1, also the exchangeability of the stationary distribution. The
chaoticity result Theorem~\ref{thm.chaos1} above is a quantitative
version of some of the results in~\cite{s91}.

\medskip

To conclude this section, we mention that analogues of results
in~\cite{lm04,lm04b} are proved in~\cite{lm03} for a related
balls-and-bins model, where, instead of queueing up to receive service
on a first-come first-served basis, customers (balls) have independent
exponentially distributed `lifetimes' and each departs its queue
(bin) as soon as its lifetime has expired.

\smallskip

Current work in progress~\cite{f} includes extensions of the results
in~\cite{lm04,lm04b} to the supermarket model where the number of
choices $d=d(n)$ and the arrival rate $\lambda= \lambda(n)$ are
$n$-dependent, including the interesting case where $d \to \infty$ and
$\lambda \to 1$ with various functional dependencies between $\lambda$
and $d$.

\medskip

\section{Coupling and bounded differences method generalised}

\label{sec:conc}

This section contains our main results and applications. We use the
notation introduced in Section~\ref{sec:not}.

Let us state our first theorem, which gives concentration of measure
for Lipschitz functions of a discrete-time Markov chain on state space
$S$ and with transition matrix $P$ at time $t$, under
assumptions on the Wasserstein distance between its $i$ step
transition measures for $i \le t$.
\begin{theorem}
\label{thm.conca-general}
Let $P$ be the transition matrix of a discrete-time Markov chain with
discrete state space $S$.
\begin{thmenumerate}
\item
Let $(\alpha_i: i \in \N)$ be a sequence of positive constants 
such that, for all $i$,
\begin{equation}
\label{cond-gen}
\sup_{x,y \in S: d(x,y) =1}\dw (\delta_x P^i,\delta_y P^i) \le \alpha_i.
\end{equation}
Let $f$ be a 1-Lipschitz function. Then for all $u > 0$, $x_0 \in S$, 
and $t > 0$,
\begin{equation}
\label{ineq.conca-general}
\P_{\delta_{x_0}} (|f(X_t)-\E_{\delta_{x_0}} [f(X_t)] |\ge u ) \le
2e^{-u^2/2(\sum_{i=1}^t \alpha_i^2)}.
\end{equation}
\item
More generally, let $S_0$ be a non-empty subset of $S$, and let
$(\alpha_i: i \in \N)$ be a sequence of positive constants 
such that, for all $i$,
\begin{equation}
\label{cond-gen-1}
\sup_{x,y \in S_0: d(x,y) =1}\dw (\delta_x P^i,\delta_y P^i) \le \alpha_i.
\end{equation} 
Let
$$S_0^0= \{ x \in S_0: y \in S_0 \mbox{ whenever } d(x,y)=1\}.$$
Let $f$ be a 1-Lipschitz function. Then for all $x_0 \in S_0^0$, $u >0$ and $t >0$,
\begin{equation}
\label{ineq.conca}
\P_{\delta_{x_0}} \Big (\{|f(X_t)-\E_{\delta_{x_0}} [f(X_t)] |\ge u \}\cap \{X_s
\in S_0^0 : \mbox{ } 0 \le s \le t\} \Big )\le
2e^{-u^2/2(\sum_{i=1}^t \alpha_i^2)}.
\end{equation}
\end{thmenumerate}
\end{theorem}

If the Markov chain becomes contractive after a finite number of
steps, then one can deduce from Theorem~\ref{thm.conca-general}
concentration results for the stationary measure of the Markov chain,
as in the following corollary.
\begin{corollary}
\label{cor.conca-general}
\begin{thmenumerate}
\item
Suppose that there exists $x \in S$ and a sequence $\alpha_i : S \to
\R^+$ of functions such that, for all $y \in S$,
\begin{equation}
\label{cond-gen-2}
\dw (\delta_x P^i,\delta_y P^i) \le \alpha_i (y),
\end{equation}
where $\alpha_i (y) \to 0$ as $i \to \infty$ for each $y$, and
\begin{equation}
\label{cond-gen-2a}
\sup_k \E_{\delta_x} [\alpha_i(X_k)]= \sup_k (P^k \alpha_i)(x) 
\to 0 \quad \mbox{ as } i \to \infty.
\end{equation}
Then $(X_t)$ has a
unique stationary measure $\pi$ and, for all $y \in S$, $\delta_y P^t
\to \pi$ as $t \to \infty$. 
\item
Suppose that~(\ref{cond-gen}) holds, and 
the constants $\alpha_i$ in
Theorem~\ref{thm.conca-general} satisfy $\sum_i \alpha_i^2 <\infty$. 
Suppose further there exists $x \in S$ such that 
$$\sup_k (P^k g)(x) < \infty,$$
where $g(y) = d(x,y)$.
Then
$(X_t)$ has a unique stationary measure $\pi$, and $\delta_y P^t \to \pi$
as $t \to \infty$ for each $y$. 

Furthermore, let $X$ be a stationary copy of $X_t$. Then, for all $u > 0$, and
uniformly over all 1-Lipschitz functions $f$,
 \begin{equation}
\label{ineq.conca-gen-stat}
\P_{\pi} (|f(X)-\E_{\pi} [f(X)] |\ge 2u ) \le
2e^{-u^2/2(\sum_{i=1}^{\infty} \alpha_i^2)}.
\end{equation}
\item
Suppose that $(X_t)$ has a unique stationary measure $\pi$ and
condition~(\ref{cond-gen-1}) holds, where $\sum_i \alpha_i^2 < \infty$.
Let $x \in S_0^0$, and suppose  $\delta > 0$ and 
$t_0 > 0$ are such that $\dw (\delta_x P^{t_0},
\pi) < \delta$ and 
$$\P_{\delta_x} (X_t \in S_0^0 \mbox{ for } t \le
t_0) \ge 1-\delta.$$
Let $X$ be a stationary copy of $X_t$.
Then, for all $u \ge \delta$, uniformly over all 1-Lipschitz functions $f$,
\begin{equation}
\label{ineq.conca-gen-stat-1}
\P_{\pi} (|f(X)-\E_{\pi} [f(X)] |\ge 2u ) \le
2e^{-u^2/2(\sum_{i=1}^{t_0} \alpha_i^2)} + 2\delta.
\end{equation}
\end{thmenumerate}
\end{corollary}
\begin{proof}
\begin{thmenumerate}
\item
Consider the sequence $P_i$ of measures on $(S, \cP(S))$ given by
$P_i= \delta_x P^i$; we have, using the coupling characterisation of
the Wasserstein distance,
\begin{eqnarray*}
\dw (P_i,P_{i+k}) & = & \dw(\delta_x P^i, (\delta_x P^k) P^i) \le \sum_{y
\in S} (\delta_x P^k)(y) \dw (\delta_x P^i, \delta_y P^i)\\
& \le & \sum_{y
\in S} (\delta_x P^k)(y) \alpha_i (y) \le \sup_k \E_{\delta_x}
[\alpha_i(X_k)] \to 0
\end{eqnarray*}
as $i \to \infty$, by assumption. Thus the sequence $(P_i)$ is a
Cauchy sequence and so, since the space of probability measures on $(S,
\cP (S))$ is complete with respect to the Wasserstein distance, it
must converge to a probability measure $\pi$ on $(S, \cP (S))$. It is
obvious that this measure must be stationary for $P$.

Now, take $y \in S$, and let $Q_i = \delta_y P^i$. Then
$$\dw (P_i, Q_i ) = \dw (\delta_x P^i , \delta_y P^i) \le \alpha_i (y)
\to 0 \quad \mbox{ as } i \to \infty.$$ 
It follows that $Q_i \to \pi$ as $i \to \infty$, and so $\pi$ must be
the unique stationary measure.
\item
The assumption that $\sum_i \alpha_i^2 < \infty$ implies that
$\alpha_i \to 0$ as $i \to \infty$. 
Then it is easily seen (using the fact that the distance $d(y,z)$
between each pair $y,z$ of states in finite) that
conditions~(\ref{cond-gen-2}) and~(\ref{cond-gen-2a}) of part (i) hold
for $x$, with $\alpha_i (y) \le \alpha_i d(x,y)$, and so, as 
in~(i) one can prove that 
there exists a (necessarily unique) stationary
measure $\pi$, and that $\delta_x P^t \to \pi$ as $t \to \infty$ for
each $x \in S$. 

Let us now prove the concentration of measure result, inequality~(\ref{ineq.conca-gen-stat}).
Take some $x \in S$. Given $\epsilon > 0$, for $t$ large enough the
Wasserstein distance, and hence the total variation distance, between
$\delta_x P^t$ and $\pi$ is at most $\epsilon$. Then, for $u \ge \epsilon$
and all such $t$, by Theorem~\ref{thm.conca-general} part (i),
\begin{eqnarray*}
\P_{\pi} (|f(X)-\E_{\pi} [f(X)] |\ge 2u ) & \le &
\P_{\delta_{x}} (|f(X_t)-\E_{\delta_x} [f(X_t)] |\ge u )+ \eps \\ 
& \le & 2e^{-u^2/2(\sum_{i=1}^{\infty} \alpha_i^2)} + \eps.
\end{eqnarray*}
Here we have used the fact that 
$$|\E_{\pi} [f(X)] - \E_{\delta_x} [f(X_t)]| \le \epsilon \le u.$$
Since $\eps$ is arbitrary, the result follows.

\item 
Let 
$$A_{t_0}= \{\omega: X_t(\omega) \in S_0 \mbox{ } \forall t \in [0,t_0] \}.$$
Arguing as in~(ii), and using Theorem~\ref{thm.conca-general}
part (ii), we can write, for $u \ge \delta$,
\begin{eqnarray*}
\P_{\pi} (|f(X)-\E_{\pi} [f(X)] |\ge 2u ) & \le & 
\P_{\delta_{x}} (|f(X_{t_0})-\E_{\delta_x} [f(X_{t_0})] |\ge u )+ \delta \\ 
& \le & \P_{\delta_{x}} \Big (\{|f(X_{t_0})-\E_{\delta_x} [f(X_{t_0})] |\ge
u \} \cap A_{t_0} \Big )\\
& + & 2\delta \\
& \le & 2e^{-u^2/2(\sum_{i=1}^{t_0} \alpha_i^2)}  +  2\delta,
\end{eqnarray*}
as required.

\end{thmenumerate}
\end{proof}

\medskip

To prove Theorem~\ref{thm.conca-general}, we
shall make use of a concentration inequality from~\cite{cmcd98}.
Let $(\tilde{\Omega}, \tilde{\cF}, \tilde{\P})$  be a
probability space, with $\tilde{\Omega}$ finite. Let
$\tilde{\mathcal G} \subseteq \tilde{\cF}$ be a
$\sigma$-field. Given a bounded random variable $Z$ on
$(\tilde{\Omega}, \tilde{\cF}, \tilde{\P})$, the {\it supremum} of
$Z$ in $\tilde{\mathcal G}$ is the $\tilde{\mathcal G}$-measurable
function given by 
\begin{equation}
\label{defn-sup}
\sup (Z| \tilde{\mathcal G}) (\tilde{\omega}) = \min_{A \in \tilde{\mathcal G}: \tilde{\omega} \in A} \max_{\tilde{\omega}'
  \in A} Z(\omega').
\end{equation}
Thus $\sup (Z)$ takes the value at
$\tilde{\omega}$ equal to the maximum value of $Z$ over the `smallest' event in
$\tilde{\mathcal G}$ containing $\tilde{\omega}$.
Since $\tilde{\Omega}$ is finite, we are
assured that the smallest event
containing $\omega$ does exist; the arguments used here would
work also in many cases where $\tilde{\Omega}$ is countably infinite.

The {\it conditional range} of $Z$ in $\tilde{\mathcal G}$,
denoted by $\ran (Z)$, is the $\tilde{\mathcal G}$-measurable function
\begin{equation}
\label{defn-cond-range}
\ran (Z \mid \tilde{\mathcal G}) = \sup (Z| \tilde{\mathcal G})+\sup (-Z| \tilde{\mathcal
  G}).
\end{equation}

Let $\{\emptyset, \tilde{\Omega} \} = \tilde{\cF}_0 \subseteq \tilde{\cF}_1
\subseteq \ldots $ be a filtration in $\tilde{\cF}$,
and let $Z_0, \ldots , $ be the martingale obtained by setting $Z_t
= \E (Z|{\tilde{\cF}}_t)$ for each $t$. 
For each $t$ let $\ran_t$ denote $\ran (Z_t|{\tilde{\cF}}_{t-1})$;
by definition, $\ran_t$ 
is an $\tilde{\cF}_{t-1}$-measurable function.
For each $t$, let the {\it sum of squared
  conditional ranges} $R_t^2$ be the random variable $\sum_{i=1}^t \ran^2_i$,
and let the {\it maximum sum of squared conditional ranges} $\hat r_t^2$
be the supremum of the random variable $R_t^2$, that is
$$\hat r_t^2 = \sup_{\tilde{\omega} \in \tilde \Omega} R_t^2 (\tilde \omega).$$

The following result is 
Theorem 3.14 in~\cite{cmcd98}.

\begin{lemma}
\label{thm.mart}
Let $Z$ be a bounded random variable on a probability space $(\tilde{\Omega},
\tilde{\cF}, \tilde{\P})$ with $\tilde{\E} (Z) = m$. Let
$\{\emptyset, \tilde{\Omega} \} = \tilde{\cF}_0 \subseteq
\tilde{\cF}_1 \subseteq 
\ldots \subseteq \tilde{\cF}_t$ be a filtration in $\tilde{\cF}$.
Then for any $u \ge 0$,
$$ \tilde{\P} (|Z-m |\ge u) \le 2e^{-2u^2/\hat r_t^2}.$$
More generally, for any $u \ge 0$ and any value $r_t^2$,
$$ \tilde{\P} (\{|Z-m |\ge u \} \cap \{R_t^2 \le r_t^2\}) \le 2e^{-2u^2/r_t^2}.$$
\end{lemma}

\begin{proofof}{Theorem~\ref{thm.conca-general}}
Let $f:S \rightarrow \R$ be 1-Lipschitz. Fix a time $t \in \N$,
$x_0 \in S$ and consider the evolution of $X_t$ conditional on
$X_0=x_0$ for $t$ steps, that is until time $t$. Since we have assumed
that there are only a finite number of possible transitions from any
given $x \in S$, we can build this conditional process until
time $t$ on a
finite probability space $(\tilde{\Omega}, \tilde{\cF},
\tilde{P}_{\delta_{x_0}})$: we can take $\tilde{\Omega}$ to be
the finite set of all possible paths of the process starting at time 0 in
state $x_0$ until time $t$, and $\tilde{\cF}$ to be the
power set of $\tilde{\Omega}$.

In the conditional space, for each time $j= 0, \ldots, t$, let $\tilde{\cF}_j=
\sigma (X_0, \ldots , X_j)$, the $\sigma$-field generated by $X_0,
\ldots , X_j$; so $\tilde{\cF}_0= \{\emptyset, \tilde{\Omega}\}$ and
$\tilde{\cF}_t = \tilde{\cF}$. 
We write $\E$ instead of $\tilde{\E}$ in what follows to
lighten the notation.

Consider the random variable $Z=
f(X_t): \tilde{\Omega} \to \R$. Also, for $j =0, \ldots, t$ let $Z_j$
be given by
$$Z_j = \E [f(X_t)|\tilde{\cF}_j] = \E_{\delta_{x_0}}
[f(X_t)| X_0, \ldots , X_j]= (P^{t-j}
f)(X_j),$$
where we have used the Markov property in the last equality.

Fix $1 \le j \le t$; we want to upper bound $\ran_j= \ran (Z_j \mid
\tilde{\cF}_{j-1})$. 
Fix also
$x_1, \ldots , x_{j-1} \in S$, and for $x \in S$ consider
\begin{eqnarray*}
g(x) 
 & = & \E [f(X_t)|X_j=x] =  \E [f(X_{t-j})|X_0=x]\\
& = & (P^{t-j}f)(x).
\end{eqnarray*}
Note that $Z_j(\tilde{\omega})\in \{g(x):
d(x,x_{j-1}) \le 1\}$ for $\tilde{\omega}$ such that $X_{j-1}
(\tilde{\omega}) = x_{j-1}$. It follows that, for such $\tilde{\omega}$, 
$$\ran_j (\tilde{\omega})
=  \sup_{x,y: d(x,x_{j-1}) \le 1, d(y,x_{j-1}) \le 1} |g(x)-g(y)|.$$

Let us prove part~(i) of the theorem. 
As $f$ is 1-Lipschitz,
\begin{eqnarray*}
\sup_{x,y: d(x,y) \le 2}|g(x)-g(y)| & = & \sup_{x,y: d(x,y)\le 2}
|(P^{t-j}f)(x)-(P^{t-j}f)(y)|\\
& = & \sup_{x,y: d(x,y)\le 2}|
\E_{\delta_xP^{t-j}}(f)-\E_{\delta_{y} P^{t-j}}(f)| \\
& \le & 2 \sup_{x,y: d(x,y)\le 1} |\E_{\delta_xP^{t-j}}(f)-\E_{\delta_{y}
P^{t-j}}(f)| \\
& \le & 2\sup_{x,y: d(x,y)\le 1}\dw (\delta_x P^{t-j},\delta_{y}
P^{t-j})\\
& \le & 2\alpha_{t-j},
\end{eqnarray*}
by assumption. We deduce that $\ran_j(\tilde{\omega}) \le 2
\alpha_{t-j}$ for all $\tilde{\omega} \in \tilde{\Omega}$.
It follows that
$$ \hat{r}^2_t(\tilde{\omega}) \le 4 \sum_{r=0}^{t-1} \alpha^2_{t-r},$$
uniformly over $\tilde{\omega} \in \tilde{\Omega}$.
Part (i) of Theorem~\ref{thm.conca-general} now follows from
Lemma~\ref{thm.mart}.

To prove~(ii), observe that the bound
$$\ran_j (\tilde{\omega})= \ran (Z_j \mid \tilde{\cF}_{j-1})
(\tilde{\omega}) 
\le 2 \alpha_{t-j}$$
still holds on the event $A_t=\{\tilde{\omega}: X_j (\tilde{\omega}) \in S_0^0 \mbox{ for } j=0,
\ldots, t\}$.
\end{proofof}

\medskip

The following special case of model satisfying the hypotheses of
Theorem~\ref{thm.conca-general} is of particular interest and 
has received considerable attention in computer science literature;
see for instance~\cite{bd97,dggjm01,j98}.
Suppose~(\ref{cond-gen}) is satisfied with $\alpha_i = \alpha^i$, where
$0 < \alpha < 1$ is a constant. In the language of~\cite{bd97} this
corresponds to the following situation. 
Consider different copies
$(X_t),(X'_t)$ of the process with initial states $x,x'$ respectively,
that is $X_0=x$ and $X'_0=x'$ almost surely. Suppose that we can couple
$(X_t),(X'_t)$ so that,
uniformly over all pairs of states $x,x' \in S$ with $d(x,x') =1$,
$$\E [d(X_1,X'_1)| X_0=x, X'_0=x'] \le \alpha ,$$
for a constant $0 < \alpha < 1$. Thus, under the coupling,
$(X_t),(X'_t)$ will be getting closer and closer together on average
as $t$ gets larger, which implies strong mixing properties~\cite{bd97,j98}.
Then, uniformly over
$x,x' \in S$ with $d(x,x') =1$, $\dw (\delta_x P,\delta_{x'} P) \le
\alpha$. By `path coupling'~\cite{bd97,j98}
$$ \E [d(X_1,X'_1)| X_0=x, X'_0=x'] \le \alpha d(x,x'),$$
and hence $\dw (\delta_x P,\delta_{x'} P) \le \alpha
\dw(\delta_x,\delta_{x'})$ for all pairs $x,x' \in S$.
By induction on $t$,
$$\dw (\delta_x P^t,\delta_{x'} P^t)\le \alpha^t d(x,x')$$
for all $x,x' \in S$ and all $t \in
\N$. Then, in the same notation as earlier, we can upper bound
$$\hat{r}^2 \le
4 \sum_{r=1}^t \alpha^{2r} \le 4 \alpha^2 (1-\alpha^2)^{-1},$$
for all $t$. Hence we obtain the following corollary.
\begin{corollary}
\label{cor.equilibrium} Suppose that there is a constant $0 < \alpha
< 1$ such that
\begin{equation}
\label{cond.1} \dw (\delta_x P,\delta_{x'} P) \le \alpha
\end{equation}
for all $x,x' \in S$ such that $d(x,x')=1$. Then for all $t > 0$
\begin{equation}
\label{ineq.equilibrium} \P_{\delta_{x_0}}
(|f(X_t)-\E_{\delta_{x_0}} [f(X_t)] |\ge u ) \le
2e^{-u^2(1-\alpha^2)/2 \alpha^2}
\end{equation}
for all $u > 0$, all $x_0 \in S$, and for every 1-Lipschitz
function on $S$.

Hence, if $X$ has the equilibrium distribution $\pi$ then, for all
$u > 0$ and every 1-Lipschitz function $f$,
\begin{equation}
\label{ineq.equilibrium-1} \P_{\pi} (|f(X)-\E_{\pi} [f(X)] |\ge u
) \le 2e^{-u^2(1-\alpha^2)/2 \alpha^2}
\end{equation}
\end{corollary}

The particular choice of $\alpha = 1- c_1/n$ for a constant $c_1 > 0$
corresponds to the `optimal' mixing time $O(n \log n)$ for a Markov
chain in a system with size measure $n$, and gives concentration of
measure in equilibrium of the form
\begin{equation}
\label{ineq.equilibrium-1a} \P_{\pi} (|f(X_t)-\E_{\pi} [f(X_t)] |\ge u
) \le 2e^{-u^2/c_2 n},
\end{equation}
where $c_2>0$ is a constant. This is the case, for example, for the
subcritical ($\beta <1$) mean-field Ising model discussed in Section~\ref{sec:examples} -- see
for example~\cite{lpw} or~\cite{llp08} for a description of the
coupling that implies fast decay of the Wasserstein distance. The same
also applies to the Glauber dynamics for colourings on bounded-degree
graphs analysed in~\cite{dgm00} (see also~\cite{dggjm01}
and~\cite{m99}). The application is straightforward when the number of colours $k$ is 
greater than $2 D$, where $D$ is the maximum degree of the
graph. It is only a little more involved in the case $(2-\eta )D \le k
\le 2D$, where the proof in~\cite{dgm00} relies on {\em delayed
path-coupling}~\cite{ckkl99}, whereby a new Markov chain is used with one step
corresponding to $c n$ steps of the original one, $n$ being the size
of the graph to colour. 

On the other hand $\alpha = 1-6/(n^3-n)$ for
the Glauber dynamics on linear extensions of a partial order of size
$n$~\cite{bd97,j98} gives an upper bound $O(n^3 \log n)$ on mixing. The
corresponding bound on deviations of a 1-Lipschitz function from its mean of size
$u$ is of the form $2e^{-u^2/cn^3}$, which is useless. However, one
cannot do much better in general. To see this, consider the partial
order on $n$ points consisting of a chain of length $n-1$ and a single
incomparable element. It is not hard to check that in this case the
mixing time is of the order $n^3$ -- see~\cite{bd97} for details. It
is also easy to see  that there is no normal
concentration of measure in the sense of~(\ref{ineq.equilibrium-1a}).

\medskip

We shall now apply Theorem~\ref{thm.conca-general} and
Corollary~\ref{cor.conca-general} to the supermarket
process described in Section~\ref{sec:balls-and-bins}, or rather to the corresponding
discrete-time jump chain $\hat{X}_t$. Recall that, when in state $x$, 
the next event is an arrival with probability $\lambda/(1+
\lambda)$, and is a potential departure with probability $1/(1+ 
\lambda)$. Given that the next event is an arrival, the queue to which
the arrival will go is determined by selecting a uniformly random
$d$-tuple of queues and sending the customer to a shortest one 
among those chosen, ties being split by always going to the first best
queue in the list. Given that the next event is a potential departure, the
departure queue is chosen uniformly at random among the $n$ possible
queues, and departures from empty queues are ignored.
In the Markov chain graph, two states are connected by an edge if
and only if they differ exactly in one customer in one queue. Then a
function $f$ is 1-Lipschitz if and only if it is 1-Lipschitz with
respect to the $\ell_1$ distance on the state space $S$. 

We focus on the case $d\ge
2$. For $d=1$, in equilibrium the queue lengths are independent
geometric random variables, so normal concentration of measure can be
obtained using the standard bounded differences inequality~\cite{cmcd98}.

By Lemma 2.3 in~\cite{lm04}, for all $x,y \in S$ such that $d(x,y)=1$, and all $t \ge 0$,
$$\dw (\delta_x P^t, \delta_y P^t) \le 1.$$

Let $c$ be a positive constant, and let $S_0$ be given by
$$S_0=\{x \in S: \parallel x \parallel_1 \le cn, \parallel x
\parallel_{\infty} \le c \log n\}.$$
It is very easy to modify the proof of 
Lemma 2.6 in~\cite{lm04} to show that,
if $x, y \in S_0$ and $d(x,y)=1$, then for some constants $\alpha,
\beta >0$,
\begin{equation}
\label{eqn.superm-contract}
\dw (\delta_x P^t, \delta_y P^t) \le e^{-\beta t/n} + 2 e^{-\beta n}
\end{equation}
for $t \ge \alpha n \log n$. 

Take a constant $K > 2$ and let $\tau = n^K$.
Then we can put $\alpha_i = 1$ for $t \le \alpha n \log n$, and
$\alpha_i = e^{-\beta t/n} + 2 e^{-\beta n}$ for $\alpha n \log n < t
\le \tau$. Then
for $t \le \tau$, we can upper bound 
$$\sum_{i=1}^{t} \alpha_{i}^2 \le \min \{t, \alpha  n \log n +
n^{1-\beta/\alpha} \beta^{-1}+ 2 e^{-\beta n/2}\} \le \min \{t, 2
\alpha n \log n\}.$$

Consider the all-empty state, ${\bf 0} \in S_0^0$.  
 Then by choosing the constant $c$ in
the definition of $S_0$ sufficiently large, we can ensure that, for $d
\ge 2$,
$$\P_{\bf 0}( \hat{X}_t \in S_0^0 \mbox{ } \forall \mbox{ } t \le \tau)\ge
1- e^{-(\log n)^2/c}.$$
This follows from Lemma 2.3 (monotone coupling for given $n$ and $d$), Lemma 2.4 (a) 
and the monotone coupling for given $n$ and different $d,d'$ (see the
proof of Lemma 2.4 in~\cite{lm04}) and equation~(37) 
in~\cite{lm04}. (See also the statements of these results in
Section~\ref{sec:balls-and-bins}.)

By Theorem~\ref{thm.conca-general} (i), we can choose $c$
sufficiently large so that, for all $t>0$, all $u >0$, and every
Lipschitz function $f$, 
\begin{equation}
\label{ineq.superm-conc}
\P_{\delta_{{\bf 0}}} (|f(\hat{X}_t)-\E_{\delta_{{\bf 0}}} [f(\hat{X}_t)] |\ge u ) \le
2e^{-u^2/c t}.
\end{equation}
By Theorem~\ref{thm.conca-general} (ii),  
for $\alpha n \log n \le t \le \tau$, and all $u >0$,
\begin{equation}
\label{ineq.superm-conc-1}
\P_{\delta_{{\bf 0}}} (|f(\hat{X}_t)-\E_{\delta_{{\bf 0}}}
[f(\hat{X}_t)] |\ge u ) \le 
2e^{-u^2/\alpha n \log n} + e^{-(\log n)^2/c}.
\end{equation}
In particular, for $\alpha n \log n \le t \le \tau$, and $u \le c_0
\sqrt{n} \log n$, 
\begin{equation}
\label{ineq.superm-conc-1a}
\P_{\delta_{{\bf 0}}} (|f(\hat{X}_t)-\E_{\delta_{{\bf 0}}}
[f(\hat{X}_t)] |\ge u ) \le 
2e^{-u^2/c n \log n},
\end{equation}
provided that $c$ is large enough.
Inequalities~(\ref{ineq.superm-conc}) -- (\ref{ineq.superm-conc-1a})
improve on what one could obtain for the jump chain 
from Lemma~\ref{lem.conc} above, for an interesting range of $u$ and
$t$ -- and it is easy to use them to derive improved concentration of
measure inequalities for the continuous chain also. (It is possible to
optimise inequality~(\ref{ineq.superm-conc-1}) by playing with the
definition of $S_0$ to obtain normal concentration
for larger $u$.)

We now want to relate this to concentration of measure in equilibrium,
via Corollary~\ref{cor.conca-general}.
It is easy to see from earlier
work (see~\cite{lm04} and references therein) that the
supermarket jump chain has a unique stationary measure. (This could
also be proven showing that the hypotheses of
Corollary~\ref{cor.conca-general} (i) are satisfied,
via~(\ref{eqn.superm-contract}) above.)

By Lemma 2.1
in~\cite{lm04} and  straightforward calculations for the Poisson
process, there is a constant $\eta > 0$ such that
\begin{equation}
\label{ineq.empty-wass}
\dw (\cL (\hat{X}_t,{\bf 0}), \hat{\pi}) \le n e^{-\eta t/n} + 2cn \P_{\hat{\pi}}
(M>\eta t/n) + 2 e^{-\eta n},
\end{equation}
where $M$ denotes the maximum queue length in equilibrium, and we may
take $c$ the same as in the definition of $S_0$, assuming that $c$ is
sufficiently large. Thus,
by~(\ref{ineq.empty-wass}),
$$\dw (\cL (\hat{X}_{\tau}, {\bf 0}), \hat{\pi}) \le (n+2cn + 2)e^{-\eta n}.$$
Let $\hat{Y}$ denote the queue lengths
vector in equilibrium.
It then follows  by Corollary~\ref{cor.conca-general}~(iii),
uniformly for
all 1-Lipschitz functions $f$, for $u \ge 1$ and $n$ sufficiently large
\begin{equation}
\label{ineq.superm-equil}
\P_{\hat{\pi}} (|f(\hat{Y})-\E_{\hat{\pi}} [f(\hat{Y})] |\ge 2u ) \le 2e^{-u^2/c n \log
n} + 2e^{-(\log n)^2/c}.
\end{equation}
So, choosing $c$ to be sufficiently large, for all $u >0$ and $n$
sufficiently large,
\begin{equation}
\label{ineq.superm-equil-1}
\P_{\hat{\pi}} (|f(\hat{Y})-\E_{\hat{\pi}} [f(\hat{Y})] |\ge 2u ) \le
ce^{-u^2/c n \log 
n} + ce^{-(\log n)^2/c}.
\end{equation}

\medskip

This improves on Lemma~\ref{lem.conca} above, and gives normal
concentration for $u = O(n^{1/2} (\log n)^{3/2})$ (again, it is
possible to obtain normal concentration for larger $u$), but is not the optimal
result we are after. In particular, we still cannot show that deviations of size
$n^{1/2} \omega (n)$ have probability tending to 0 for $\omega (n)$
tending to infinity arbitrarily slowly. We will now derive another inequality that will
enable us to achieve our aim. 

\begin{theorem}
\label{thm.concb-general}
Assume that there exists a set $S_0$ and numbers $\alpha_i (x,y)$ ($x,y
\in S_0$, $i \in \N$) such that, for all $i$, and all $x,y \in S_0$
with $d(x,y)=1$,
\begin{equation}
\label{cond-gen-3}
\dw (\delta_x P^i, \delta_y P^i) \le
\alpha_i (x,y).
\end{equation}
Let 
$$S_0^0 = \{x \in S_0: y \in S_0 \mbox{ whenever } d(x,y)=1 \}.$$
For $x \in S$, let $g_x(y) = \dw (\delta_y P^i, \delta_x P^i)^2$. 
Assume that, for some sequence $(\alpha_i: i \in \N)$ of positive
constants,
\begin{equation}
\label{cond-gen-4}
\sup_{x_0 \in S_0^0} (P g_{x_0})(x_0)  \le \alpha_i^2.
\end{equation}
Let $t >0$, let $v = \sum_{i=1}^t \alpha_i^2$, and let
\begin{equation}
\label{cond-gen-5}
\hat{\alpha} = \sup_{1 \le j \le t} \sup_{x,y \in S_0: d(x,y)\le 2}
\alpha_j (x,y).
\end{equation}
Let also $A_t= \{\omega: X_s(\omega) \in S_0^0 : \mbox{ } 0 \le s \le t\}$.

Then, for all $u >0$, and uniformly over all 1-Lipschitz functions $f$,
\begin{equation}
\label{ineq.concc}
\P_{\delta_{x_0}} \Big (|f(X_t)-\E_{\delta_{x_0}} [f(X_t)] |\ge u
\cap A_t \Big )\le
2e^{-u^2/(4v(1+(\hat{\alpha}u/6v))}.
\end{equation}
\end{theorem}

\medskip

To prove Theorem~\ref{thm.concb-general}, we use another result from~\cite{cmcd98}.
With notation as before, for $j=1, \ldots, t$, let 
$$\var_j = \var (Z_j \mid \tilde{\cF}_{j-1})= \E \Big ((Z_j - \E(Z_j \mid
\tilde{\cF}_{j-1}))^2 \mid \tilde{\cF}_{j-1} \Big );$$
let $V= \sum_{j=1}^t \var_j$. Also, for each such $j$, let $\dev_j = \sup
(|Z_j-Z_{j-1}| \mid \tilde{\cF}_{j-1})$, and let $\dev = \sup_j \dev_j$. 
The following result is 
essentially Theorem 3.15 in~\cite{cmcd98}.

\begin{lemma}
\label{thm.mart-b}
Let $Z$ be a random variable on a probability space $(\tilde{\Omega},
\tilde{\cF}, \tilde{\P})$ with $\E (Z) = m$. Let
$\{\emptyset, \tilde{\Omega} \} = \tilde{\cF}_0 \subseteq
\tilde{\cF}_1 \subseteq 
\ldots \subseteq \tilde{\cF}_t$ be a filtration in $\tilde{\cF}$.
Let $\hat b = \max \dev$, the maximum conditional deviation (and
assume that $\hat b$ is finite). 
Then for any $u \ge 0$,
$$ \P (|Z-m |\ge u) \le 2e^{-u^2/(2 \hat{v}(1+(\hat bu/3 \hat{v}))},$$
where $\hat{v}$ is the maximum sum of conditional variances (which is
assumed to be finite). 

More generally, for any $u \ge 0$ and any values $b,v \ge 0$,
$$ \P (\{|Z-m |\ge u \} \cap \{V \le v\} \cap \{\max \dev \le b \}) 
\le 2e^{-u^2/(2v(1+(bu/3v))}.$$
\end{lemma}

\medskip

\begin{proofof}{Theorem~\ref{thm.concb-general}}
The proof is similar to the proof of Theorem~\ref{thm.conca-general}.
Let $f:S \rightarrow \R$ be 1-Lipschitz. Fix a time $t \in \N$,
an $x_0 \in S$ and consider the evolution of $X_t$ conditional on
$X_0=x_0$ for $t$ steps, that is until time $t$. Again this
conditional process can be supported by a 
finite probability space $(\tilde{\Omega}, \tilde{\cF},
\tilde{\P}_{\delta_{x_0}})$.

As before, in the conditional space, for each time $j= 0, \ldots, t$ let $\tilde{\cF}_j=
\sigma (X_0, \ldots , X_j)$, the $\sigma$-field generated by $X_0,
\ldots , X_j$; so $\tilde{\cF}_0= \{\emptyset, \tilde{\Omega}\}$ and
$\tilde{\cF}_t = \tilde{\cF}$.  Again, we
consider the random variable $Z=
f(X_t): \tilde{\Omega} \to \R$. And, for $j =0, \ldots, t$,  $Z_j$
is given by
$$Z_j = \E [f(X_t)|\tilde{\cF}_j] = \E_{\delta_{x_0}}
[f(X_t)| X_0, \ldots , X_j]= (P^{t-j}
f)(X_j).$$
 Suppose first for simplicity
that $S_0 = S$. 
We want to apply Lemma~\ref{thm.mart-b} and for this we need to
calculate the conditional variances $\var_j$. To do this, we use the
fact that the variance of a random variable $Y$ is equal to $\frac12 \E
(Y-\tilde{Y})^2$, where $\tilde{Y}$ is another random variable with
the same distribution as $Y$ and independent of $Y$.

Fix $j$ and $x_1, \ldots , x_{j-1} \in S$, and for $x \in S$ consider
\begin{eqnarray*}
g(x) 
 & = & \E [f(X_t)|X_j=x] =  \E [f(X_{t-j})|X_0=x]\\
& = & (P^{t-j}f)(x).
\end{eqnarray*}

Then, for $\tilde{\omega}$ such that $X_{j-1}(\tilde{\omega}) =
x_{j-1}$, $Z_j (\tilde{\omega}) \in \{g(x): d(x,x_{j-1}) \le 1\}$, so that
\begin{eqnarray*}
\var_j (\tilde{\omega}) & = & \frac12 \sum_{x,y} P(x_{j-1},x)
P(x_{j-1},y)(g(x)-g(y))^2\\
& \le & \frac12 \sum_{x,y:d(x_{j-1},x) \le 1,d(x_{j-1},y)\le 1} P(x_{j-1},x)
P(x_{j-1},y)\dw (\delta_x P^{t-j}, \delta_y P^{t-j})^2 \\
& \le & 2 \sum_{x: d(x_{j-1},x) \le 1} P(x_{j-1},x) \dw (\delta_x P^{t-j},
\delta_{x_{j-1}}P^{t-j})^2\\
& \le & 2 \sum_x P(x_{j-1},x) \alpha_{t-j} (x_{j-1},x)^2\\
& \le & 2 \alpha_{t-j}^2,
\end{eqnarray*}
by assumption~(\ref{cond-gen-4}).

Then we can upper bound the sum
$$\hat{v} \le 2 \sum_{j=1}^t \alpha_j^2.$$
It remains to bound $\dev= \sup_j \dev_j$. For $\tilde{\omega}$
such that $X_{j-1}(\tilde{\omega}) = x_{j-1}$,
\begin{eqnarray*}
\dev_j (\tilde{\omega}) & \le & \sup_{x:
d(x,x_{j-1})\le 1}|g(x)-(P^{t-j+1}f)(x_{j-1})|\\
& = & \sup_{x:
d(x,x_{j-1})\le 1}|(P^{t-j}f)(x)-(P^{t-j+1}f)(x_{j-1})|\\
& \le & \sup_{x: d(x,x_{j-1})\le 1}| \dw (\delta_x P^{t-j}, \delta_{x_{j-1}}P^{t-j+1}).
\end{eqnarray*}
It follows that, for each $j=1,
\ldots, t$,
\begin{eqnarray*}
\dev_j & \le & \sup_{x,y:d(x,y)\le 1}\dw (\delta_x P^{t-j+1}, \delta_y
P^{t-j}) \\
& \le &
\sup_{x,y:d(x,y) \le 2}\dw ((\delta_x P)P^{t-j}, \delta_y P^{t-j}) \\
& \le & \hat{\alpha},
\end{eqnarray*}
by~(\ref{cond-gen-5}) and using the coupling characterisation of the
Wasserstein distance.
Theorem~\ref{thm.concb-general} now follows from the first statement in
Lemma~\ref{thm.mart-b} in the case where $S_0=S$.
In general, the above bounds on $\hat{v}$ and $\dev$ hold
on the event $A_t=\{\omega: X_j (\omega) \in S_0^0 \mbox{ for } j=0,
\ldots, t\}$, and so Theorem~\ref{thm.concb-general} also follows from
the second statement of Lemma~\ref{thm.mart-b}.
\end{proofof}

\medskip

Let us now apply Theorem~\ref{thm.concb-general} to the supermarket
model from~\cite{lm04} discussed above. Again, we focus on the case $d
\ge 2$.

Let $c$ be a positive constant, and let $S_0$ be given by
$$\{x \in S: \ell(k,x)= \sum_{r=1}^n {\bf 1}_{x(r)\ge k}\le n e^{-k/c} \mbox{ for } k=1,
\ldots \}.$$
Consider the all-empty state, ${\bf 0} \in S_0^0$.
Let $K > 2$ be a constant. We claim that we can choose $c$
sufficiently large that, if $\tau=n^K$, then
$$\P_{\bf 0}( \hat{X}_t \in S_0^0: t \le \tau)\ge 1- e^{-(\log n)^2/c}.$$
This follows easily from Lemma~\ref{lem.customers.2} in
the present paper, together with equation~(\ref{eq.max-queue}).

We now want to calculate the quantity in~(\ref{cond-gen-4}). For a
state $x_0 \in S_0^0$ and a state $x$ chosen with probability
$P(x_0,x)$, these states will only differ in a queue of length greater
than $k$ if $P(x_0,x)$ is a probability of an event involving a queue
of length at least $k$ -- a departure from a queue of length at least
$k$ or an arrival into a queue of length at least $k$. For $x_0 \in
S_0^0$ such a transition happens with probability at most $c e^{-k/c}$
(choosing $c$ large enough again).

The proof of Lemma 2.6 in~\cite{lm04} shows that, 
if $x, y \in S_0$ are adjacent and differ in a
queue of length $k$, then for some constants $\alpha, \beta >0$
we can upper bound
$$\dw (\delta_x P^t, \delta_y P^t) \le e^{-\beta t/n} + 2 e^{-\beta n}$$
for $t \ge \alpha k n$. Also, by Lemma 2.3 in~\cite{lm04},
$$\dw (\delta_x P^t, \delta_y P^t) \le 1$$
for all $t$ and hence for $t < \alpha k n$.

Combining the above observations and choosing $\alpha >1$ large enough,
we find that for $t \ge \alpha^2 n$
\begin{eqnarray*}
\sup_{x_0 \in S_0^0} \E_{\delta_{x_0}} \dw (\delta_{X_1} P^t,
\delta_{x_0} P^t)^2 \le e^{- t/\alpha n} + e^{- n/\alpha}.
\end{eqnarray*}

Hence, by choosing $c$ large enough, we can upper bound
$$\sum_{i=1}^{\tau} \alpha_{i}^2 \le c  n
.$$
Further, once again using Lemma 2.3 in~\cite{lm04}, we can upper bound
$\hat{\alpha} \le 2$.

By
Theorem~\ref{thm.concb-general}, there is a constant
$c>0$ such that, uniformly for
all 1-Lipschitz functions $f$, all $t \le \tau$, and all $u > 0$,
\begin{equation}
\label{ineq.superm-conc-2}
\P_{\delta_{{\bf 0}}}
(|f(\hat{X}_t)-\E_{\delta_{{\bf 0}}} [f(\hat{X}_t)] |\ge u )
 \le  
2e^{-u^2/4c (n+u)} + e^{- (\log n)^2/c}.
\end{equation}
In particular, we can choose $c$ large enough so that, for $u \le c_0 \sqrt{n} \log n$,
\begin{equation}
\label{ineq.superm-conc-2a}
\P_{\delta_{{\bf 0}}}
(|f(\hat{X}_t)-\E_{\delta_{{\bf 0}}} [f(\hat{X}_t)] |\ge u )
 \le  
3e^{-u^2/cn}.
\end{equation}

Now, as before, by~(\ref{ineq.empty-wass}),
$$\dw (\delta_{{\bf 0}} P^{\tau}, \hat{\pi}) \le (n+2cn +2)e^{-\eta n}$$
provided $c$ is large enough. 
It follows that for $n$ large enough, uniformly for
all 1-Lipschitz functions $f$, and all $u \ge 1$,
\begin{eqnarray}
\P_{\hat{\pi}} (|f(\hat{Y})-\E_{\hat{\pi}} [f(\hat{Y})] |\ge 2u ) & \le &
\P_{\delta_{{\bf 0}}}
(|f(\hat{X}_{\tau})-\E_{\delta_{{\bf 0}}} [f(\hat{X}_{\tau})] |\ge u )
\nonumber \\
& + & (n+2cn
+2)e^{-\eta n} \nonumber \\
& \le &  
2e^{-u^2/4c (n+u)} + 2e^{- (\log n)^2/c} 
\label{ineq.superm-conc-3}
\end{eqnarray}
It follows that, for $0 < u \le c_0 n^{1/2} \log n$, we obtain
\begin{eqnarray}
\P_{\hat \pi} (|f(\hat{Y})-\E_{\hat \pi} [f(\hat{Y})] |\ge 2u ) & \le & 
ce^{-u^2/cn},
\label{ineq.superm-conc-3a}
\end{eqnarray}
provided that the constant $c$ is chosen sufficiently large.
Choosing $u = \sqrt{n} \omega (n)$, where $\omega (n)$
is a function tending to infinity with $n$ arbitrarily slowly, we
obtain
$$\P_{\hat \pi} (|f(\hat{Y})-\E_{\hat \pi} [f(\hat{Y})] |\ge u )=o(1)$$
as $n \to \infty$.

Inequalities~(\ref{ineq.superm-conc-2}) and~(\ref{ineq.superm-conc-3})
could be optimised (by optimising the choice of set $S_0$) 
to obtain normal concentration for larger $u$.

\medskip

For a positive integer $k$, let $\ell (k,\hat{Y})$ be the number of queues
of length at least $k$ in the stationary jump chain, and let $\hat{\ell}
(k)$ be its expectation. Then for any positive integer $s$, and any $u>0$, we can write
\begin{eqnarray*}
\E_{\hat \pi} [|\ell (k,\hat{Y})-\hat{\ell}(k)|^s] \le u^s + \sum_{y \ge
u}y^{s-1} \P_{\hat \pi} (|\ell (k,\hat{Y})-\hat{\ell} (k)| >y).
\end{eqnarray*}
Note that the maximum value that $|\ell (k,\hat{Y})-\hat{\ell} (k)|^s$ can take is
$n^s$. Then, taking $u=n^{1/2}$, and
applying inequality~(\ref{ineq.superm-conc-3}), we obtain
\begin{eqnarray*}
\E_{\hat \pi} [|\ell (k,\hat{Y})-\hat{\ell}(k)|^s] \le c n^{s/2}.
\end{eqnarray*}
assuming the constant $c$ is chosen big enough. Hence, arguing as
in Section~4 of~\cite{lm04}, it is easy to show that
$$\sup_k |\E[\ell (k,\hat{Y})^r-\hat{\ell} (k)^r| =O(n^{r-1}).$$
And hence, arguing as in Section 5 of~\cite{lm04}, we obtain that, for
some constant $c_0$,
\begin{equation}
\label{eq.equi-mean}
\sup_i |n^{-1}\hat{\ell} (i) - \lambda^{1+d+...+d^{i-1}} | \le c_0 n^{-1}, 
\end{equation}
which improves on equation~(27) in~\cite{lm04}, implying that
$$\sup_i |n^{-1}\hat{\ell} (i) - \lambda^{1+d+...+d^{i-1}} | \le c_0 n^{-1}
(\log n)^2.$$

\medskip

\section{Conclusions}

\label{sec:conclusions}

We have derived concentration inequalities for Lipschitz functions of a
Markov chain long-term and in equilibrium,  
depending on contractivity properties of the chain in question.
Our results apply to many natural Markov chains in computer science
and statistical mechanics.
 
One open problem is to show that, in a discrete-time Markov chain with `local' 
transitions, under suitable conditions,
rapid mixing occurs essentially if and only if there is normal 
concentration of measure long-term and in equilibrium (with
non-trivial bounds). 
Another open question is to explore how these
properties relate to the cut-off phenomenon. Is it the case that,
again under suitable assumptions, 
they are  necessary and sufficient conditions for a cut-off
to occur?

\section{Acknowledgement}

The author is grateful to Graham Brightwell for reading the paper and
making many helpful comments.

\medskip

\newcommand\AAP{\emph{Adv. Appl. Probab.} }
\newcommand\JAP{\emph{J. Appl. Probab.} }
\newcommand\JAMS{\emph{J. \AMS} }
\newcommand\MAMS{\emph{Memoirs \AMS} }
\newcommand\PAMS{\emph{Proc. \AMS} }
\newcommand\TAMS{\emph{Trans. \AMS} }
\newcommand\AnnMS{\emph{Ann. Math. Statist.} }
\newcommand\AnnPr{\emph{Ann. Probab.} }
\newcommand\CPC{\emph{Combin. Probab. Comput.} }
\newcommand\JMAA{\emph{J. Math. Anal. Appl.} }
\newcommand\RSA{\emph{Random Struct. Alg.} }
\newcommand\ZW{\emph{Z. Wahrsch. Verw. Gebiete} }
\newcommand\DMTCS{\jour{Discr. Math. Theor. Comput. Sci.} }

\newcommand\AMS{Amer. Math. Soc.}
\newcommand\Springer{Springer}
\newcommand\Wiley{Wiley}

\newcommand\vol{\textbf}
\newcommand\jour{\emph}
\newcommand\book{\emph}
\newcommand\inbook{\emph}
\def\no#1#2,{\unskip#2, no. #1,} 
\newcommand\toappear{\unskip, to appear}

\newcommand\webcite[1]{
\texttt{\def~{{\tiny$\sim$}}#1}\hfill\hfill}
\newcommand\webcitesvante{\webcite{http://www.math.uu.se/~svante/papers/}}
\newcommand\arxiv[1]{\webcite{arXiv:#1.}}

\def\nobibitem#1\par{}

\end{document}